\documentclass[a4paper]{amsart}
\usepackage[utf8]{inputenc}
\usepackage{amsmath}
\usepackage{graphicx}
\usepackage{amsfonts}
\usepackage{amssymb}
\usepackage{caption}
\usepackage{float}
\usepackage{subfigure}
\usepackage{color}
\usepackage{geometry}

\def\BBox{\kern  -0.2cm\hbox{\vrule width 0.2cm height 0.2cm}}

\newtheorem{theorem}{Theorem}
\newtheorem{proposition}{Proposition}
\newtheorem{lemma}{Lemma}
\newtheorem{corollary}[theorem]{Corollary}
\newtheorem{definition}{Definition}

\newtheorem{example}{Example}
\newtheorem{observation}{Remark}

\begin{document}
\title[Classical properties of algebras using a new graph association.]
{Classical properties of algebras using a new graph association}

\author{R.M. Aquino, L.M. Camacho, E.M. Cañete, C. Cavalcante, A. Márquez}

\address{[Aquino R.M.] Dpto. de Matemática. Universidade Federal do Espirito Santo. Vitoria (Brasil) } \email{aquino.ufes@gmail.com}

\address{[Camacho L.M. --- Márquez A.] Dpto. Matem\'{a}tica Aplicada I.
Universidad de Sevilla. Avda. Reina Mercedes, s/n. 41012 Sevilla.
(Spain)} \email{lcamacho@us.es --- almar@us.es}

\address{[Cañete E.M.] Instituto Matemática. Universidade Federal de Alagoas. Maceió (Brasil)}
\email{elisa.mat@gmail.com }

\address{[Cavalcante C.] Instituto de Ciências Matemáticas e Computaç\~{a}o. Universidade de S\~{a}o Paulo. S\~{a}o Paulo (Brasil)}
\email{himehimur@gmail.com}

\thanks{The authors were supported by MCTI/CNPQ/Universal 10/2014, CAPES, and individually L.M.Camacho by MTM2016-79661-P (European FEDER support included) and A. Márquez by BFU2016-74975-P}

\maketitle
\begin{abstract} We study the relation between algebraic structures and Graph Theory. We have defined five different weighted digraphs associated to a finite dimensional algebra over a field in order to tackle important properties of the associated algebras, mainly the nilpotency and solvability in the case of Leibniz algebras.
\end{abstract}



\vspace{2,0cm}

 The study of the relationship between algebraic and discrete structures is fruitful. Numerous contributions to this topic have made, for instance, in Graph Theory, one may consider matrices associated to a graph in order to obtain properties of those graphs as a consequence of the study of the spectrum of the adjacency matrix or the Laplacian matrix of such graph. One may also recall the case of Cayley's graphs associated to groups. We refer to \cite{Godsil-Royle} for further details. It is also well known the classification of semi-simple Lie algebras by  Dynkin diagrams, for the case of finite dimensional Lie algebras over algebraically closed fields (see \cite{Dynkin}). In \cite{ARS} the studies of quiver algebras and representation theory of finite dimensional associative algebras over a field are shown. Finally it is worthwhile to point out Carriazo, Fern\'{a}ndez and N\'{u}\~{n}ez's paper \cite{Carriazo},  where the authors have developed a method to associate a discrete structure to a finite dimensional Lie algebra that lead them to presente interesting applications concerning solvability and nilpotency of those algebras. The problem of that work is that the authors strongly applied the identity of Jacobi as a fundamental tool to obtain their results, thus their results are limited to Lie algebras, in fact, they are limited to a specificaly family of Lie algebras. 

It is natural to ask if it is possible to extend the relationship between any algebra and a graph, in order to obtain algebraic properties of the algebra by looking the properties of the graph. The aim of this work is to answer this question in the affirmative. Specially, in this paper we are going to show that given a finite dimensional algebra over a field (and a fixed basis), it is possible to associate some weighted oriented finite graphs, which allow us to recover completely the algebra, i.e.,to recover the product structure of the algebra. But the main contribution of this work is that if we drop the information about the labels (obtaining thus more simple digraphs), one obtains interesting algebraic properties of the algebra by analysing the properties of these simple digraphs. Furthermore, we succeed to apply our results to an important family of algebras: the Leibniz algebras, obtaining in this case significant results about nilpotency and solvability. 

We remark that Leibniz algebras rise up in the context of the studies of homology of Lie algebras, mainly on the studies of the  Chevalley–Eilenberg boundary map, we refer to \cite{LP} for the reader. Lie algebras are  related to Lie groups which are closely related to Yang-Mills' theory, one of the best reason to study Lie algebras hence Leibniz algebras.

It is worthwhile to point out that the graph and results obtained in \cite{Carriazo} have been encompassed by our work. Moreover we have obtained properties of the digraphs by analysing their associated algebras (see section 4 for more details).

 Throughout this work, we fix as algebra any  finite dimensional algebra over an infinite fixed field $\mathbb{F}$ and as digraph any oriented finite graph. In Section 1 we recollect the  main definitions and results useful to introduce our results. In Section 2 we define two classes of digraphs  associated to a finite dimensional algebra $A$ over a field, with a fixed basis and a given list of relations.  We also show that the digraphs associated to Lie algebras in \cite{Carriazo} are a particular case of the digraphs defined in our current work, in Subsection 2.1. In Section 3 some basic properties of general algebras and their associated digraphs are given. In Section 4 interesting relationships between algebraic properties (as stability of derived series and the lower central series) and digraphs are shown. A complete discussion of nilpotency and solvability of Leibniz algebras is given in that section, in particular, in Theorem \ref{nilpotent} , Proposition \ref{Nilpiff} and Theorem \ref{solvable}. In Section 5 some new digraphs are introduced in order to preserve more information about the algebraic and discrete structures and achieve more general results. 
  Finally, some interesting counterexamples have been presented to provide the strength of our hypotheses on those results.

\section{Preliminaries}

Throughout this paper $\mathbb{F}$ will denote an algebraically closed field with characteristic zero.

We remark that any graph $G$ throughout this paper is directed in the sense to be a digraph. Particularly, we will denote a (directed) graph $G$ by a pair of sets $ (V(G), A(G)) $ where $ V(G)$ is the non-empty set of vertices and $A(G)$ is a set of directed edges or arrows, presented by ordered pairs of vertices $x_i$ and $x_j$, denoted by $(x_1,x_2)$. We notice that hypothesis fix as two the number of arrows between two fixed vertices, one arrow to each direction. In this context, we will call two arrows between the same pair of vertices by parallel arrows or parallel edges.

We recall that the outdegree (respectively the indegree) of a vertex $v \in G$ (denoted $\delta_{G}^{\rm out}(v)$, or, respectively, $\delta _{G}^{\rm in}(v)$) is the number of arrows in $G$ starting at $v$ (resp. ending at $v$). 

Throughout this work, we consider the definition of cycles in graphs as usual, meaning, a closed path over that graph (we ignore the orientation of the arrows). We will denote a cycle with length $n$ by  $n$-cycle. When necessary, we will consider cycles over oriented graphs, that is, a closed oriented walk over the digraph. Furthermore to fix our context we may call those cycle as oriented cycle.

It is well known that every digraph with no oriented cycles has a topological sorting, that is, it admits an ordering of the vertices such that the starting endpoint of every arrow occurs earlier in the ordering than the ending endpoint of the edge. It is also called an {\it admissible order} of the digraph (see \cite{Godsil-Royle}).
	
Given a digraph G, we define a digraph $H$ such that $ V(H) \subseteq V(G) $ and $ A(H)\subseteq A(G) $ subdigraph of $G$. This subdigraph is said induced by $ X\subseteq V(G) $ if $V(H)=X$ and every arrow between elements of $X$ appearing in $A(G)$ appears in $A(H)$ as well.

Furthermore, we will call two digraphs $G$ and $ H $  isomorphic if there is a bijective function $f:V(G) \to V(H) $ satisfying that there exists $ (v_1, v_2) \in A(G)$ if and only if, there exists $(f(v_1), f(v_2)) \in A(H) .$ We will denoted $G\simeq H.$

From the next definitions we follow \cite{Dies}. A vertex colouring of a graph $G$ is a map $c: V \to S$ such that $c(v) \neq c(w)$ whenever $v$ and $w$ are adjacent. The elements of the set $S$ are called the available colours and $c$ is called colour map. The set of all vertices with one colour is independent and is called a "colour class". An $n$-colouring of a graph $G$ uses $n$ colours; it thereby partions $V$ into $n$ colour classes. The chromatic index of $G,$ denoted by $\chi(G)$, is defined as the minimum $n$ for which $G$ has an $n$-colouring. Analogously, we have the edge colouring of a graph. 


Similarly one may define a function from the set of arrows of a graph to a set of numbers, for instance, natural, real or complex numbers. This function and the graph associated to it are called weight function and weighted graph respectively.


We recall that an algebra $A$ over a field  $\mathbb{F}$ is a ring $ A = ( A, +, *)$ such that $(A,+)$ has a structure as vector space over $\mathbb{F}$. We will denote shortly by  $A $ or by $ (A,*)$ when it is necessary to identify the product of that algebra.

Let  $A$ be an $n$-dimensional algebra  over the field $\mathbb{F}$. Let $\{ x_1,$ $ x_2,$ $\cdots,$ $x_n\}$ be a fixed basis of the vector space $(A,+).$ Then $A$ is determined, up to isomorphisms, by the multiplication rule for the basis elements: 
$x_i*x_j=\displaystyle\sum_{k=1}^n \gamma_{ij}^k x_k,$ for $\gamma_{ij}^k \in \mathbb{F} .$  We will take the same notations as used in Lie algebras by calling those elements, the $\gamma_{ij}^k$, as the structure constants of such algebra. Therefore, for a fixed  basis of $(A,+)$, we can regard each algebra of dimension $n$ over a field $\mathbb{F}$ as a point in the $n^3$-dimensional space of structure constants endowed with the Zariski topology. Recall that the Zariski topology of an algebraic variety is the topology whose closed sets are the algebraic subsets of the variety. In the case of an algebraic variety over the complex numbers, the Zariski topology is thus coarser than the usual topology, as every algebraic set is closed for the usual topology.

As usual in Lie theory, the sequences 
$$\begin{array}{ccc}
A^1=A, & A^{k+1}=A^k*A^1, & k \geq 1,\\ 
A^{[1]}=A, & A^{[s+1]}=A^{[s]}* A^{[s]}, & s\geq 1,
\end{array}$$
are said to be {\it the lower central series} and {\it the derived series of $A,$} respectively. We remark that those concepts can be applied to any algebra.

We recall that an algebra $(A,[-,-])$ over a field  $\mathbb{F}$ is called a Leibniz algebra if for any $x,y,z\in A$, the {\it Leibniz identity}
$$	Leib(x,y,z) : \big[[x,y],z\big]-\big[[x,z],y\big]-\big[x,[y,z]\big] =0 $$ holds. In presence of anti-commutativity, Leibniz identity becomes Jacobi identity and therefore Lie algebras are examples of Leibniz algebras. From the Leibniz identity we conclude that for any $x, y\in A$ the elements $[x,x]$ and $[x,y]+[y,x]$ lie in $Ann_r(A)$, the right annihilator of the algebra $A$, defined as usual. Let $g$ and $h$ be two Leibniz algebras, a homomorphism of Leibniz algebras is a $\mathbb{F}$-linear map $\phi: g \longrightarrow h $ such that $\phi ([x,y]_{g})=[\phi(x),\phi (y)]_{h},$ for all $x,y \in g.$ 

A Leibniz algebra $(A,[-,-])$ is called nilpotent (respectively, solvable) if there exists  $m\in\mathbb{N} $ ($t\in \mathbb{N}$) such that $A^m=0$ (respectively, $A^{[t]}= 0$). The minimal number $m$ (respectively, $t$) such that this property is valid is called the nilindex (respectively, index of solvability) of the algebra $A.$
	

\section{Digraphs associated  to algebraic structures} \label{associatedgraphs}

In this section we will define two different weighted oriented finite graphs associated to  a $\mathbb{F}$-algebra $(A,*)$  of dimension $n$ with a fixed basis $\{x_1,\dots, x_n\}$ as follows.

\begin{definition} Let the algebra $(A,*)$ over a field  $\mathbb{F}$ and $\{x_1,\dots, x_n\}$ its basis.

\begin{itemize}

\item {\label {definition 1} Action to the right digraph:} We  denoted by $G_R(A)$ the following  weighted oriented graph as described below:

\begin{enumerate}
\item The vertices of the digraph are $\{x_1, \dots, x_n\}.$ 
\item For each $i \in \{1, \dots,n \},$ we consider the products $x_i*x_j=\displaystyle\sum_{l=1}^n c_{ij}^l x_l$ different to zero for $j \in \{1, \dots, n \}.$ Hence there exists an edge $\alpha_{il}$, oriented from $x_i$ to $x_l$, if $c_{ij}^l \neq 0,$ for $j=1,\dots,n.$ 
\item The weight of the edge $\alpha_{il}$ is the vector $(c_{ij}^l)_j.$ 
\end{enumerate}

\item We define dually the action to the left digraph, by considering the action on the left and we will denoted that digraph by $G_L(A).$

\end{itemize}

\end{definition}

\begin{observation} 
\rm
We notice  that those digraphs are well defined for any finite dimensional $\mathbb{F}$-algebra and even for locally finite dimensional algebras.
\end{observation}

\begin{example} {\it Group algebras.} 
\rm
Let $G$ be a finite group of order 2 and the group algebra $A= \mathbb{F} G$  associated to $G$ for $ \mathbb{F}$ being the complex number field. We fix a basis for $A,$ $\{ g_1, g_2\},$ where $g_1$ and $g_2$ are the elements of the group $G$ such that $ g_1*g_j = g_1$ and $g_2 * g_i =  g_i$ for any $i,j = 1,2.$ Hence we have the associated digraphs to $A$ shown in the following figures.

\begin{figure}[H] 
\label{fig:example1}
\begin{minipage}[b]{0.25\textwidth}
\includegraphics[width=0.9\textwidth]{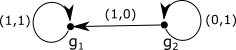}
\caption{$G_R(A)$}
\end{minipage}
\begin{minipage}[b]{0.25\textwidth}
\includegraphics[width=0.9\textwidth]{fig_example1_2.png}
\caption{$G_L(A)$}
\end{minipage}
\end{figure}

\end{example}

\begin{example} \label{exampledim2} {\it 2-dimensional Leibniz algebras.}
\rm

In this example we will present a complete non-redundant list of digraphs (previosly defined) associated with a 2-dimensional Leibniz algebra over the field of the complex numbers. In order to achieve that, we need the following two results.

\begin{proposition}
Let $A$ be a 2-dimensional Leibniz algebra over $\mathbb{C}.$ Then $A$ is isomorphic either the abelian algebra, denoted by $A^1$ or to one of the non isomorphic algebras:
$$\begin{array}{ccc}
 A^2:\begin{cases}
 [x_1,x_1]=x_2,
 \end{cases} &
A^3: \begin{cases} 
[x_1,x_2] =x_1,
 \end{cases} & 
A^4: \begin{cases}
 [x_1,x_2]=x_2,\\
 [x_2,x_1]=-x_2.
 \end{cases}
\end{array}$$
\end{proposition}

 As the isomorphism between algebras is not transmitted to isomorphism between digraphs, we are going to show all 2-dimensional Leibniz algebras, that can be obtained from the above algebras by change of bases. 

\begin{corollary}
Let $A$ be a 2-dimensional Leibniz algebra over $\mathbb{C}.$ Then either $A$ is the abelian algebra $A^1$ or it is one of the following algebras:
$$\small
\begin{array}{ccc}
 A^2:\begin{cases}
 [x_1,x_1]=x_2,
 \end{cases} &
A^3: \begin{cases} 
[x_1,x_2] =x_1,
 \end{cases} & 
A^4: \begin{cases}
 [x_1,x_2]=x_2,\\
 [x_2,x_1]=-x_2.
 \end{cases} \\
 A^{2,1}:\begin{cases}
 [x_2,x_2]=x_1,
 \end{cases} &
A^{2,2}: \begin{cases} 
 [x_1,x_1]=-b x_1+\frac{b^2}{a} x_2,\\
 [x_1,x_2]=-a x_1+b x_2,\\
 [x_2,x_1]=-a x_1+b x_2,\\
 [x_2,x_2]=-\frac{a^2}{b} x_1+a x_2.\\
 a,b\neq 0
 \end{cases} & 
A^{3,1}: \begin{cases}
  [x_1,x_1]= x_2,\\
 [x_2,x_1]= x_2.
  \end{cases}\\
 A^{3,2}:\begin{cases}
  [x_2,x_2]= x_1,\\{}
 [x_1,x_2]= x_1.
 \end{cases} &
A^{3,3}: \begin{cases} 
 [x_2,x_1]=x_2.
 \end{cases} & 
A^{3,4}: \begin{cases}
   [x_2,x_1]=-d x_1+c x_2,\\
 [x_2,x_2]=-\frac{d^2}{c} x_1+d x_2,\\
 c,d\neq 0.
  \end{cases}\\ 
   A^{3,5}:\begin{cases}
  [x_1,x_1]=-\frac{eg}{f} x_1+e x_2,\\
 [x_1,x_2]=-\frac{eg^2}{f} x_1+\frac{eg}{f} x_2,\\
 [x_2,x_1]=-g x_1+f x_2,\\
 [x_2,x_2]=-\frac{g^2}{f} x_1+g x_2,\\
 e,f,g\neq 0,\ g(f^2-eg)\neq 0.
 \end{cases} &
A^{3,6}: \begin{cases} 
 [x_1,x_2]=i x_1-h x_2,\\
 [x_1,x_1]=h x_1-\frac{h^2}{i} x_2,\\
 h,i\neq 0.
 \end{cases} & 
A^{4,1}: \begin{cases}
   [x_2,x_1]= x_1,\\
 [x_1,x_2]=- x_1.
  \end{cases}\\
  \\
A^{4,2}: \begin{cases}
   [x_1,x_2]=-j x_1-k x_2,\\
 [x_2,x_1]=j x_1+k x_2,\\
    j,k\neq 0.
  \end{cases}\\   
\end{array}$$
\end{corollary}

 We notice that those classifications will follow straightforward from solving the space of solutions of the system of equations given by the structure constants and the equations provided by the Leibniz identity. 

\begin{observation}
Note that the algebra $A^{2,1}$ and the algebras of the family $A^{2,2}$ are isomorphic to $A^2,$; the algebras $A^{3,1},$ $A^{3,2},$ $A^{3,3}$ and the algebras of the families $A^{3,4},$ $A^{3,5}$ and $A^{3,6}$ are isomorphic to $A^{3};$ and $A^{4,1}$ and the algebras of the family $A^{4,2}$ are isomorphic to $A^4.$ 
\end{observation}

The table below shows the action to the right digraphs associate to the 2-dimensional Leibniz algebras. Note that in this example we have omitted the weight of the digraphs. Thus, any 2-dimensional Leibniz algebra has associated as action to the right digraph one of the showing.

 \begin{table}[h]
\centering
\begin{tabular}{|c|c|c|c|} \hline
& & & \\
$A^1$ & $A^2$ & $A^3$ & $A^4$\\
\includegraphics[width=0.10\paperwidth]{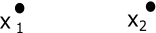} 
& 

\includegraphics[width=0.10\paperwidth]{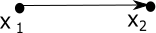} 
& 

\includegraphics[width=0.11\paperwidth]{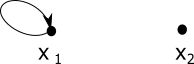} 
& 

\includegraphics[width=0.11\paperwidth]{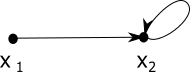} 
\\ \hline 
& & & \\
 & $A^{2,1}$ & $A^{3,1}$ & $A^{4,1}$ \\
 & \includegraphics[width=0.10\paperwidth]{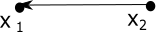} 
& 

\includegraphics[width=0.11\paperwidth]{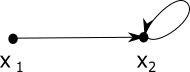} 
& 

\includegraphics[width=0.11\paperwidth]{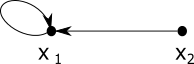} 
\\ \hline 
& & & \\
& $A^{2,2}$ & $A^{3,2}$ &$A^{4,2}$ \\
& \includegraphics[width=0.13\paperwidth]{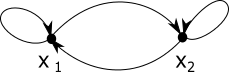} 
& 

\includegraphics[width=0.11\paperwidth]{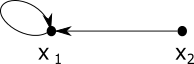} 
& 

\includegraphics[width=0.11\paperwidth]{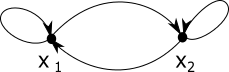} 
\\  \hline 
 & & & \\
& & $A^{3,3}$ & \\
 & & \includegraphics[width=0.11\paperwidth]{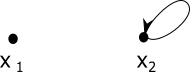} 
&  \\   \hline 
 & & & \\
& & $A^{3,4}$ & \\
 & & \includegraphics[width=0.11\paperwidth]{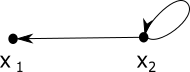} 
&  \\  \hline 
 & & & \\
& & $A^{3,5}$ & \\
 & & \includegraphics[width=0.13\paperwidth]{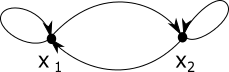} 
&  \\  \hline 
 & & & \\
& & $A^{3,6}$ & \\
 & & \includegraphics[width=0.11\paperwidth]{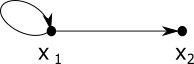} 
 & \\ \hline  
 \end{tabular}
\caption{Classes of the action to the right grahps of 2-dimensional Leibniz algebras}
\label{tab:classes_2dimLeibniz}
\end{table}

\end{example}

\subsection{Remarks on Carriazo-Fern\'{a}ndez-N\'{u}\~{n}ez's graphs }

One of the motivations of this work was founded in \cite{Carriazo} which leads us to improve the link between Lie algebras and Graph Theory given by the authors in that paper. The digraphs defined above  in {\it Definition 1} show us that the link between Lie algebras and Graph Theory can be generalized for any algebraic structure. Furthermore if we consider the action to the left or to the right digraphs, results concerning solvability and nilpotency have been obtained  for general classes of algebras, by considering the algebraic structure and the properties of the digraphs instead of using the structure constants. We refer to Section 3, 4 and 5 to the reader.

In fact, if we assume that the algebra $A$ has the same hypothesis given in \cite{Carriazo}, that is, $A$ is a Lie algebra defined by the law  $[x_i,x_j]=a_ix_i+b_jx_j, \hbox{ for }a_i,b_j \in \mathbb{F}, $
and we consider the action to the right digraphs $G_R(A),$ similar results like the obtained in \cite{Carriazo} can be achieved by following theirs proofs. We notice that the most important tool used to prove those results is the Jacoby identity, thus we follow the same steps to obtain the same kind of results as given by the authors in \cite{Carriazo}.

Precisely, given a Lie algebra $A$ under the hypotheses considered in \cite{Carriazo}, both the digraph $G_R(A)$ and $G_L(A)$ are a generalization of the graph presented in \cite{Carriazo}, provides that they represent the relationships between each operand and the result of the product, the same relationships investigated in \cite{Carriazo}. Under these hypothesis, the operands and the result of product lead us to the same elements, therefore the digraphs contains the same information.  

We remark that for any Lie algebra $A,$ $G_R(A) \simeq G_L(A)$, see Proposition \ref{lie}. Furthermore, recognizing that $G_R(A)$ is equivalent to digraphs in \cite{Carriazo} except for loops occurring at the going-in vertex and zero coordinates in weight of digraphs action, we can see that equivalent results to those found in \cite{Carriazo} can be obtained by considering these digraphs.

In order to clear up that the digraphs defined in previous section are the generalization of the graphs defined in \cite{Carriazo}, we show as examples the following results, by denoting into parentheses the corresponding results of  \cite{Carriazo}. It is worthwhile to note that in the following results of this subsection we are going to consider $A$ as is a Lie algebra defined by the law  $[x_i,x_j]=a_ix_i+b_jx_j, \hbox{ for }a_i,b_j \in \mathbb{F}.$

\begin{lemma} (Lemma 3.1 in \cite{Carriazo})
Let $A$ be a Lie algebra associated to a digraph $G_R(A).$ Then, the configurations of Figure \ref{fig:tipo2_perm} are allowed in $G_R$ for any three different vertices $i,j,k.$

\begin{figure}[H]
		\includegraphics[width=0.5\paperwidth]{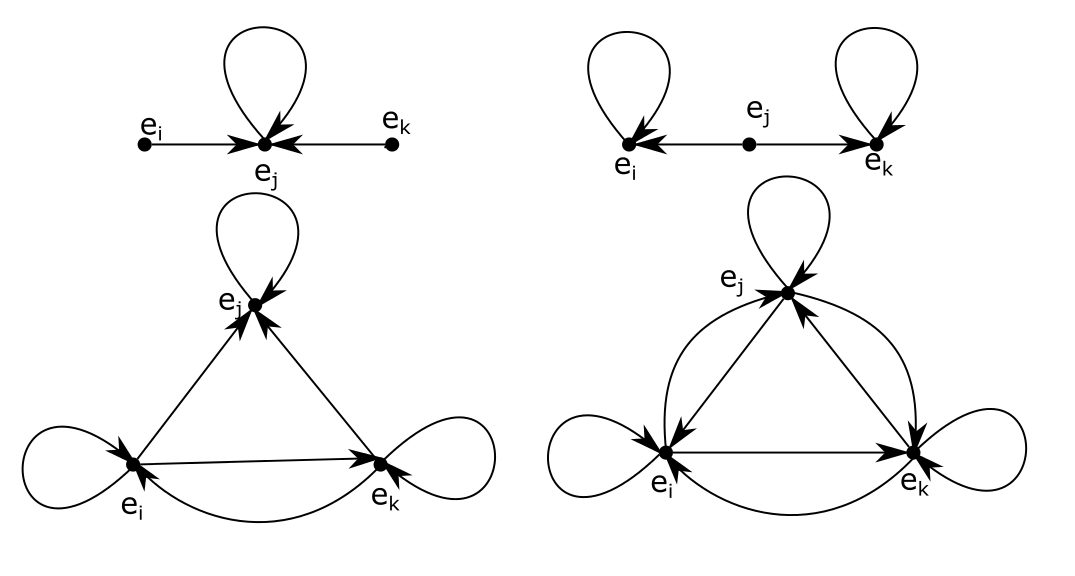} 		
		\caption{Allowed configurations in action to the right digraph.}
		\label{fig:tipo2_perm}
\end{figure}

\end{lemma}

We observe  that each edge of the digraphs from \cite{Carriazo} appears in $G_R(A)$. Furthermore, double edges in \cite{Carriazo} creates in $G_R(A)$ two loops, one for each vertex linked by double edge. The weight in \cite{Carriazo} graphs is equivalent to the weights in $G_R(A)$ except for the signal, that is opposite in edges directed from the vertex with the highest index to lowest.

\begin{theorem} \label{teo:tipo2_theo3.2_carr}( Theorem 3.2 in \cite{Carriazo}) 
Let $G_R(A)$ be the action to the right digraph associated with a Lie algebra $A$. If $G_R(A)$ has no 3-cycles, then $G_R(A)$ is given as follow.
\begin{enumerate}
\item Two vertices $a$,$b$, two oriented edges with opposite directions and one loop to each vertex.
\item A well-oriented graph with loops in incoming vertex.
\end{enumerate}
Furthermore, any digraph satisfying one of those conditions above is associated with a Lie algebra since that the one single coordinate nonzero of weights provided to each edge be the coordinate that corresponds to incoming vertex.
\end{theorem}

\begin{theorem}(Theorem 3.6 in \cite{Carriazo})
Let $G_R(A)$ be a digraph containing 3-cycles associated with a Lie algebra. Then the following conditions holds:
\begin{enumerate}
\item The double edges of $G_R(A)$ lie on 3-cycles and there are no 3-cycles without double edges. 
\item The adjacent vertices to the extreme vertices of the double edges are not mutually adjacent. Moreover, they appear in one of the configurations of Fig. \ref{fig:tipo2_3c_perm}.
\item The subdigraph obtained from $G_R(A)$ by removing its double edges and loops attached to it satisfies condition (2) of Theorem \ref{teo:tipo2_theo3.2_carr} .
\end{enumerate}
\end{theorem}

\begin{figure}[H]
		\includegraphics[width=0.5\linewidth]{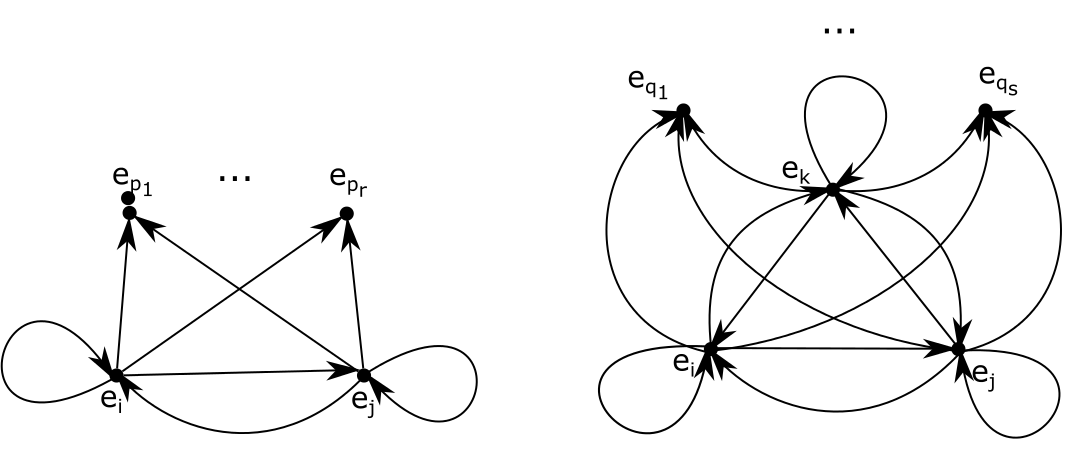} 
		\caption{Allowed configurations with 3-cycles in digraph action to the right }
		\label{fig:tipo2_3c_perm}
\end{figure}

\begin{theorem}(Theorem 4.1 in \cite{Carriazo})
Any Lie algebra associated with a weighted digraph that does not contain the second configuration of Figure \ref{fig:tipo2_3c_perm} is solvable, but not nilpotent.
\end{theorem}

\section{Algebras and digraphs}

In section 2 we have associated two digraphs with an algebra but, as we saw in Example \ref{exampledim2},  isomorphic algebras can be associated to non isomorphic digraphs under these definitions. Hence, the following natural question rises up: when a digraph determines an algebraic structure? In this section, we will present basic relations between an algebra and their associated digraphs to attempt to answer this question.

\begin{proposition}\label{prop-inverse}
Let $G$ be a non-weighted digraph. Then, there exists an algebra whose action to the right digraph (or action to the left digraph) is isomorphic to $G$ with an associated weight.
\end{proposition}
\begin{proof}
Let us see the proof for the action to the right digraph, for the action to left digraph is analogous.

Given a digraph $G$, define $A$ as the linear envelope of $V(G)$. From the definition of linear envelope, $A$ is a vector space and $V(G)$ is a basis of that vector space. Now, for each arrow $(x_i,x_j) \in A(G)$, we define the following multiplication $x_i*x_j = x_j$ in $A$. It is clear that if we extend that operation by bilinearity we will obtain an algebra such that its associated action to the right digraph is isomorphic to $G$ with the corresponding weight.
\end{proof}

Some basic properties related to an algebra and its associated digraphs are given by the following result.

\begin{lemma}
Let $A$ be an algebra, and let  $G_R(A)$, and $G_L(A)$ be their associated digraphs. Then, the following properties hold:

\begin{itemize}
\item An element $x$ of the basis of $A$ belongs to $ Ann(A)$ if and only if $\delta^{out}_{G_R(A)}(x)=\delta^{out}_{G_L(A)}(x)=0.$
\item An element $x$ of the basis of $A$ belongs to $Ann_r(A)$ if and only if $\delta^{out}_{G_L(A)}(x)=0.$
\end{itemize}

\end{lemma}

\begin{proof}
The proof follows straightforward from the definitions of the annihilator of $A,$ the right annihilator of $A,$ and the considered digraphs.
\end{proof}

\begin{definition}
    Let $A$ be an algebra and $\mathcal{B} = \{x_1, ..., x_n\}$ a basis to $A$. We will use the notation $G_R(A)^2$ to indicate the subdigraph of $G_R(A)$ induced by the vertices that belong to $W = \{x_k \in \mathcal{B}: \exists \ y \in A^2 \text{ such that } y = \sum_{i=1}^{n} \alpha_i x_i \text{ with } \alpha_k \neq 0 \}$. Respectively, $G_L(A)^2$ will be denote the subdigraph of $G_L(A)$ induced by vertices that belong to $W$.
    
    This definition is equivalent to the subdigraph induced by all the vertices with nonzero indegree in $G_R(A)$ or, respectively, in $G_L(A)$.
\end{definition}

\begin{example}{} \label{example_G^2}
\rm
    Let $A$ be the 4-dimensional algebra given by the following table of multiplication.

$$A:\begin{cases}
[x_1,x_2]= 2x_3-x_2,\\
[x_1,x_1]= - x_2,\\
[x_2,x_4]= x_1,\\
[x_4,x_1]= - x_1.
\end{cases}$$    

The action to the left and to the right digraphs and its subdigraphs $G_L(A)^2$ and $G_R(A)^2$ are presented in the next figure.

\begin{figure}[H]\centering
\begin{subfigure}[$G_L(A)$]{
\centering 
        \includegraphics[width=5cm]{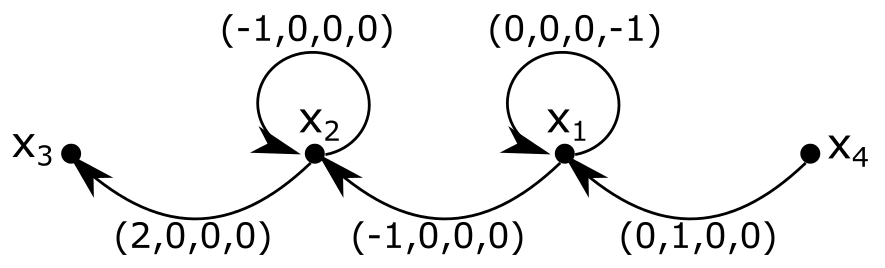}}
    \end{subfigure}
\hspace{2cm}
    \begin{subfigure}[$G_R(A)$] {\centering 
        \includegraphics[width=3.5cm]{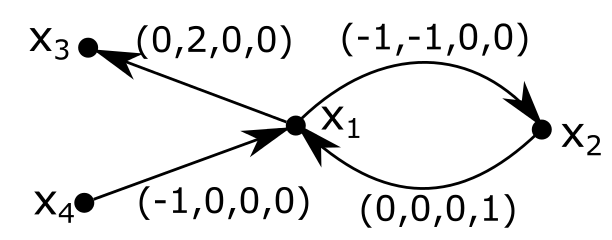}}
    \end{subfigure} 
\newline
     \begin{subfigure}[$G_L(A)^2$] {\centering 
        \includegraphics[width=4cm]{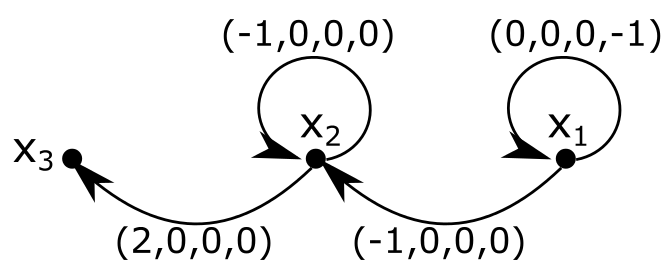}}
    \end{subfigure}
\hspace{2cm}
    \begin{subfigure}[$G_R(A)^2$]{ \centering 
        \includegraphics[width=4cm]{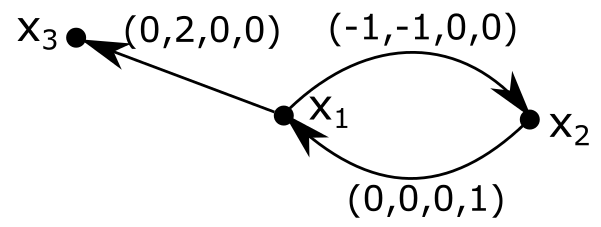}}
    \end{subfigure}
\end{figure}    
\end{example}

$G_R(A)^2$ and $G_L(A)^2$ are useful to obtain an upper and a lower bound of the chromatic index of $G_R(A)$ and $G_L(A).$
\begin{proposition}
    Let $A$ be an algebra. The following statements about its associated digraphs are true:
        \begin{eqnarray*}
            \chi(G_R(A)^2) &\leq \chi(G_R(A)) &\leq \chi(G_R(A)^2)+1 \\
            \chi(G_L(A)^2) &\leq \chi(G_L(A)) &\leq \chi(G_L(A)^2)+1
        \end{eqnarray*}
\end{proposition}

\begin{proof}
It is enough to prove the first statement, the second one follows analogously. 

The bound $\chi(G_R(A)^2) \leq \chi(G_R(A))$ is clear because $G_R(A)^2$ is a subdigraph of $G_R(A).$ 

On the one hand, from digraph colouring, if $\chi(G_R(A)) = n$, there is a complete subdigraph of $G_R(A)$ with $n$ vertices. On the other hand, note that the vertices that have only output arrows do not belong to $G_R(A)^2,$ but it is impossible that neither of two connected vertices have an input arrow. Therefore, we have a complete subdigraph with $n-1$ vertices in $G_R(A)^2,$ i.e., $\chi(G_R(A)) \leq \chi(G_R(A)^2)+1$ .
\end{proof}





\section{Non-associative algebras}

The goal of this work is to study some classical properties of algebras by analysing the structure of the associated graphs. It is worthwhile to point out that although the relationship between algebraic and discrete structures has been worked in other papers as we have referenced in the Introduction of this work, it has not been studied thoroughly the fact of obtaining algebraic properties of an algebra by analysing, mainly, the structure of its graphs.  In this section, we will consider the defined digraphs in Section~\ref{associatedgraphs} by removing their weights, to deal only with the combinatorial structure of that digraphs. With these non-weight digraph, properties such as nilpotency and solvability will be tackled. 

We recall that the nilpotency and solvability study of an algebras is a main classical problem in Algebra, particularly, in Leibniz algebras case. It is known that every finite dimensional Lie algebra is decomposed into a semi-direct sum of a semisimple subalgebra and its solvable radical, we refer  Levi's Theorem in  \cite{Jac} to the reader. Moreover,  the study of solvable Lie algebras is reduced to the study of nilpotent algebras,  see Mal'cev's results in \cite{Mal1}. Recently, Barnes in \cite{Bar} has proved an analogue of Levi's Theorem for Leibniz algebras, namely, a Leibniz algebra is decomposed into a semi-direct sum of its solvable radical and a semisimple Lie algebra.

Although Proposition~\ref{prop-inverse} proves that any digraph is associated to a certain algebra,  we have verified that there exist digraphs that can not be realised as a digraph associated to some algebra under some specific conditions. We have the following result.

\begin{proposition}
There are digraphs which are not action to the right digraph (respectively, action to the left digraph) associated to Leibniz algebras. In particular, let $G$ be the digraph with $n$ vertices and arrows given by loops, one for each vertex. Then there is no Leibniz algebra  whose associated action to the right digraph(respectively, action to the left digraph) is $G$. 
\end{proposition}

\begin{proof}
Suppose $A$ is a Leibniz algebra and $G$ (as above) its associated action to the right digraph. 
We have: 
$$[x_j,x_i]=c_{ji}^j x_j, \ 1\leq i\neq j\leq n,$$ 

where, for each $j$ it must  exist $i_j$ such that $c_{j\, i_j}^j\neq 0.$ 

By applying the Leibniz identity $Leib(x_k,x_j,x_{i_j}),$ for every $j$ and  $j+1\leq k\leq n,$ we obtain that $c_{k\, j}^k=0$ for all $k.$  Therefore we have proved that the digraph $G$ has not a loop in the vertex $x_k,$ which is a contradiction.

Considering $G$ the action to the left digraph of $A$, the proof is analogous.
\end{proof}

\begin{proposition} \label{lie}
Let $A$ be a Lie algebra. Then $G_L(A) \approx G_R(A)$.
\end{proposition}

\begin{proof} We observe that $[x,x]=0$ and $[x,y]= - [y,x]$ for $x, y \in A$ since $A$ is a Lie algebra.  

Therefore, the anty-symmetric property of the product in  $A$ tell us that  $G_L(A) \approx  G_R(A)$ since edges leaving $x_i$ reach the same set of vertices that the edges leaving $x_j$ whenever $[x_i,x_j]$ is a nonzero element of $A,$ for any pair $(x_i, x_j)$ in the fixed basis of $A$.
\end{proof}

\begin{observation} 
\rm
In general a digraph verifying the above conditions does not correspond to a Lie algebra. For instance, the non--Lie Leibniz algebra presented below satisfies those hypothesis.\end{observation}

\begin{example}{}
\rm
Let $A$ be a $3$-dimensional Leibniz algebra given by the following relations with respect to the basis $\{x_1,x_2,x_3\}:$
$$\begin{array}{ll}	
[x_1,x_2]=\alpha x_3,&\\{}
[x_2,x_1]=\beta x_3, & \alpha \beta\neq 0,\ \beta\neq -\alpha.
\end{array}$$

We have that $A$ is a non-Lie Leibniz algebra and  it satisfies that $G_L(A) \simeq G_R(A).$
\end{example}

Let us see the first result of this work that, looking into the digraph, talks about nilpotency. It is important to point out that the property required to the graph: to be a acyclic digraph, is very useful and commonly used in Graph Theory since it is  possible to check whether a given directed graph is acyclic in linear time. 

\begin{theorem}\label{nilpotent} Let $A$ be an algebra and $G_R(A)$ its associated digraph. If $G_R(A)$ has no oriented cycles then the lower central series converges to $\{0\}$. Particularly, if $A$ is a Leibniz algebra, then $A$ is nilpotent.
\end{theorem}

\begin{proof} Since $G_R(A)$ has no oriented cycle, the topological sorting is possible. Let 
$ {\mathcal{B}} = \{x_{1,1}, $ $ x_{1,2},\dots,$ $ x_{1,n_1},$ $ x_{2,1},$ $ x_{2,2},$ $\dots x_{2,n_2},\dots,$ $ x_{m,1}, \dots, x_{m,n_m}\} $ be a basis of A which respects the order of the topological sorting (see Figure \ref{fig:topologicalordering}).

\begin{figure}[H]
       
		\includegraphics[width=0.3\paperwidth]{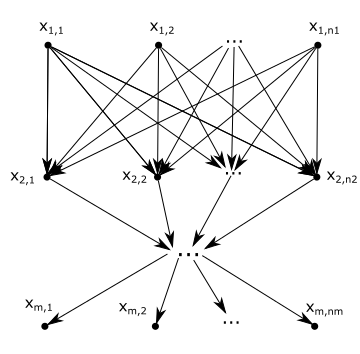} 
        \caption{Topological Ordering}
        \label{fig:topologicalordering}
\end{figure}

The law of the algebra is defined by the products  
$ x_{i,j} * x_{p,q} = \displaystyle\sum_{r=1}^{m} \left( \displaystyle\sum_{s=1}^{n_r} C_{i,j;p,q}^{r,s} x_{r,s} \right),$ for all $x_{i,j}, x_{p,q}, x_{r,s} \in \mathcal{B}.$ Moreover, due to the topological sorting, we assert that $ C_{i,j;p,q}^{r,s} = 0,\  \forall r \leq i .$ Therefore 
$ x_{i,j} * x_{p,q} = \displaystyle\sum_{r=i+1}^{m}\displaystyle\sum_{s=1}^{n_r} C_{i,j;p,q}^{r,s} x_{r,s}. $

We claim that for every $ x \in A $ such that 
$x = \displaystyle \sum_{r=1}^{m} \left( \displaystyle\sum_{s=1}^{n_r} \alpha_{r,s} x_{r,s} \right) $ with $ \alpha_{1,s} \neq 0 $ for some $s \in \mathbb{N}$, we have $ x \notin A^2.$ In fact, if $ x \in A^2 $ then $\exists$ $y,z \in A $ such that $ y*z = x $. Since one may write y, z with respect to the  elements of the basis as follow:
\begin{eqnarray*}
y &=& \sum_{r=1}^{m}( \sum_{s=1}^{n_r} \beta_{r,s} x_{r,s}) \\
z &=& \sum_{r=1}^{m}( \sum_{s=1}^{n_r} \gamma_{r,s} x_{r,s}), \\
\end{eqnarray*}
hence the element  $x=y*z$ can be written as follow:
\begin{eqnarray*}
y*z &=& [\sum\limits_{r=1}^{m} ( \sum\limits_{s=1}^{n_r} \beta_{r,s}x_{r,s} )] * [\sum\limits_{u=1}^{m} ( \sum\limits_{v=1}^{n_u} \gamma_{u,v}x_{u,v} )] \\
 &=& \sum\limits_{r=1}^{m}{\sum\limits_{s=1}^{n_r} [\sum\limits_{u=1}^{m}(\sum\limits_{v=1}^{n_u} \beta_{r,s} \gamma_{u,v}  (x_{r,s} * x_{u,v}) )]}\\
 &=& \sum\limits_{r=1}^{m}{\sum\limits_{s=1}^{n_r}(\sum\limits_{u=1}^{m}[\sum\limits_{v=1}^{n_u} \beta_{r,s} \gamma_{u,v} ( \sum\limits_{p=r+1}^{m}(\sum\limits_{q=1}^{n_p} C_{r,s;u,v}^{p,q}x_{p,q})))]}.
\end{eqnarray*}

 Thus every element  $x \in A ^2$ satisfies that the structural constants associated with $ x_{1,s} $ are equal to zero for all $ s. $

We will apply  induction on $k$ to prove that $ \forall x \in A^{k+1} $, $ x = \displaystyle\sum\limits_{i=k+1}^{m}(\displaystyle\sum\limits_{j=1}^{n_i}\gamma_{i,j}x_{i,j}).$ Since $x \in A^{k+1} $, we have $ x=y*z $ for any $ y \in A^k$ and $ z \in A.$ Hence we have the following equalities.
\begin{eqnarray*}  
x &=& \sum\limits_{i=k}^{m} (\sum\limits_{j=1}^{n_i} \alpha_{i,j} x_{i,j}) * (\sum\limits_{p=1}^{m} (\sum\limits_{q=1}^{n_p} \beta_{p,q} x_{p,q})) \\
 &=& \sum\limits_{i=k}^{m} {\sum\limits_{j=1}^{n_i} [\sum\limits_{p=1}^{m} (\sum\limits_{q=1}^{n_p} \alpha_{i,j} \beta_{p,q} (x_{i,j}*x_{p,q}))]}\\
 &=& \sum\limits_{i=k}^{m} {\sum\limits_{j=1}^{n_i} ( \sum\limits_{p=1}^{m} [\sum\limits_{q=1}^{n_p} \alpha_{i,j} \beta_{p,q} (\sum\limits_{r=i+1}^{m} (\sum\limits_{s=1}^{n_r} C_{i,j;p,q}^{r,s} x_{r,s})))]}.  \\
\end{eqnarray*}  

Therefore we have proved that the vertices of each $k$-line of the topological sorting of the digraph, that is, the elements that correspond to $\{x_{k-1,1},x_{k-1,2}, \dots, x_{k-1,n_{k-1}}\}$ have coefficients nulls for each element of $A ^ k$ represented in chosen basis. In other words, elements of $A^k$ are a linear combination of elements corresponding to lines bellow line $k-1$.

Finally, since $A$ is a finite-dimensional algebra, there is a finite number of vertices on $G_R(A)$ hence a finite number of lines in $G_R(A)$ which we have denoted by $m$, in the topological sorting. Therefore the elements from $A^{m}$ are written as linear combination of the $m$-th line elements, and these have not outgoing edges, that is, $ A^{m+1} = A^{m}*A = {0} $ as we wish to prove.
\end{proof}

\begin{corollary} \label{cor-boundnilindex}
Let $A$ be a Leibniz algebra and  $G_R(A)$ its associated digraph. If $G_R(A)$ has no oriented cycles  then the longest directed path in $G_R(A)$ is an upperbound of the nilindex of $A$.
\end{corollary}

\begin{observation}
\rm
We observe that in general the problem of finding the longest directed path in a graph is an NP--hard problem. However, if the digraph have not oriented cycles, it is known that one can solve that problem in linear time. We refer, for instance, \cite{sw}.
\end{observation}


Although there are digraphs with oriented cycles associated to algebras whose central lower series do not stabilize, our next result ensures the existence of a basis in which the digraph associated with $ A $ has not oriented cycles for any algebra $A$ with this feature. Thereby providing a characterization.


\begin{proposition}
    Let $A$ be an algebra whose lower central series stabilizes at $\{0\}$. Then there exists a basis of $A$ whose associated action on the right digraph has not oriented cycles.
\end{proposition}
\begin{proof}
There are some $m \in \mathbb{N}$ such that $A^m = \{0\}$ due to the lower central series of $A$ stabilizes at $\{0\}$. Let $n$ be the smallest of these. It is clear that $ \forall k \leq n$, $A^k \subsetneq A^{k-1}$. Therefore, the dimensions of $A^{s}$ are decreasing.\\

We will now proceed with a constructive proof.\\

Let $b_{i-1}$ be a basis of $A^{i-1}/A^{i}$, $\forall 1 < i \leq n, i \in \mathbb{N}$ and let $\mathcal{B}$ be the basis consisting of union of $b_i$. Since $A^{n-1}/A^n \cup A^{n-2}/A^{n-1}  \cup ...  \cup A^{1}/A^2 = A $, $\mathcal{B}$ is a basis of $A$.\\

Finally, it is possible define a topological ordering to the  action to the right digraph associated to $A$ and this basis with the vertices that represents elements in $ b_i$ in the $i$-th row. Then this digraph contains no oriented cycles.
\end{proof}

Our next result has established a necessary and sufficient condition to have finite dimensional nilpotent Leibniz algebras with a multiplicative basis by considering the property of having no oriented cycles (of its associated action to the right digraph). 

\begin{proposition} \label{Nilpiff}
Let $A$ be an algebra admitting a multiplicative basis and $G_R(A)$ its action to the right digraph. Then, the lower central series of $A$ converges to $\{0\}$ if and only if $G_R(A)$ has no oriented cycles. Particularly, if $A$ is a Leibniz algebra, $A$ is nilpotent if and only if $G_R(A)$ has no oriented cycles. 
\end{proposition}

\begin{proof} We denote $B= \{ x_1, x_2, \dots, x_n\} $ a fixed  basis for $A$. Suppose that the lower central series  converge to $\{0\}$. We assume that there exists a oriented cycle $C=<x_1,x_2,\dots,x_t>$ in $G_R(A).$ Hence without lost of generality one can write the following equalities.

\begin{align*}
x_i*-&=\beta_i x_{i+1}, 1\leq i\leq t-1,\\
x_t*-&=\beta_t x_1,\\
\end{align*}
where $-$ is some unknown element of $B.$
It follows that $x_i \in A^k$ for any $k \geq 1$ $ \forall i \in \{1, \cdots, t\} $. Therefore there are $k$ elements in each element of the lower central series, a contradiction with our assumption.

Conversely we suppose that $G_R(A)$ has no oriented cycles, then the proof follows directly from Theorem \ref{nilpotent}.
\end{proof}

We remark that the hypothesis on $A$ to admit a multiplicative basis is crucial in Proposition \ref{Nilpiff}. In fact, the following example show us  a nilpotent Leibniz algebra $A$ whose associated digraph $G_R(A)$ has oriented cycles.

\begin{example}{}
\rm
Let $A$ be the $4$-dimensional Lie algebra given by the following table of multiplication.
$$A:\begin{cases}
[x_1,x_3]= x_2 + x_3,\\
[x_1,x_2]= - x_2  - x_3,\\
[x_2,x_4]= x_1,\\
[x_4,x_1]= - x_2 - x_3,\\
[x_3,x_4]= - x_1.
\end{cases}$$

One can verify that the nilindex of A is 3 and its associated digraph $G_R(A)$ has oriented cycles given in Figure \ref{fig:example3}.

\begin{figure}[H]
		\includegraphics[width=0.25\textwidth]{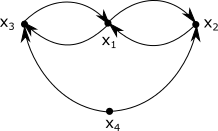} 
        \caption{Action to the right digraph of A}
        \label{fig:example3}
\end{figure}

\end{example}

Theorem \ref{nilpotent} proves that if the action to the right digraph of a Leibniz algebra $A$ is an acyclic digraph, then $A$ is nilpotent. In the next Theorem, the main result of this section,  we are going to prove that it is enough to check if either the action to the right digraph or the action to the left digraph of $A$ is acyclic in order to know if $A$ is solvable.

\begin{theorem}\label{solvable}
Let $A$ be an algebra. If either $G_R$ or $G_L$ has no oriented cycles then the derived series of $A$ converges to $\{0\}$. Particularly, if $A$ is a Leibniz algebra, then $A$ is solvable.
\end{theorem}

\begin{proof}
By considering the Theorem \ref{nilpotent}, it is enough to prove that if $G_L(A)$ has no oriented cycles then $A$ is solvable. In order to prove this statement, we  define the following sequence.
\begin{align*}
    A^{(1)} &= A, \\
    A^{(i)} &= [A,A^{(i-1)}] \ \hbox{ for } i > 1.
\end{align*}

It is clear that $A^{(1)} \subseteq  A^{[1]}.$  We claim  that $A^{(k)} \subseteq  A^{[k]}$  $ \forall k \geq 2.$ 

This proof follows analogously to the proof of Theorem \ref{nilpotent}. Precisely, by using the topological sorting of $G_L(A)$ and the sequence $A^{(n)}$, one can prove, for all $k,$ that  the elements of $A^{(k)}$ belonging to each $k$-line of the topological sorting of $G_L(A)$ have null coeficients in the linear combination representing the elements of $A^{(k+1)}$ in the specified basis. Finally, we consider  the fact that $G_L(A)$ has a finite number of vertices. Hence there exists $m>0$ such that $A^{[m+1]}=0.$

\end{proof}

\begin{observation}
\rm
 We observe that the reciprocal is not true, in general. The $4$-dimensional solvable non-nilpotent Leibniz algebra $A$ given by 
 $$A: \begin{cases}
[x_1,x_3]=x_1,\\
[x_2,x_4]=x_2,\\
[x_4,x_2]=-x_2.
\end{cases}$$
has associated digraphs with oriented cycles.

\begin{figure}[H] 
\label{fig:obs4_grafostipo2}

\begin{minipage}[c]{0.45\textwidth}
\centering{\includegraphics[width=0.4\textwidth]{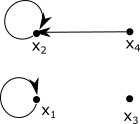}}
\caption{Action to the right digraph of $A$.}
\end{minipage}
\hfill
\begin{minipage}[c]{0.45\textwidth}
\centering{\includegraphics[width=0.4\textwidth]{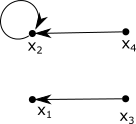}}
\caption{Action to the left digraph of $A$.}
\end{minipage}
\end{figure}

\end{observation}

\section{The union digraph}

The steps followed in this work have been firstly, to create a biunivocal relationship between weigthed digraphs with finite dimensional algebra with a fixed basis, and secondly, to analise the digraph structure in order to tackle and conclude from them some interesting algebraic properties of our algebra. 
In this section, we are going to give one step more, we are going to be able to obtain algebraic properties of the algebra by analysing the structure of the subjacent graph (without neither orientation or weight). In order to achieve that goal, we introduce new digraphs and sequence.

\begin{definition}
Let $(A,*)$ be an algebra and $\{x_1,\dots, x_n\}$ a basis of $A$.

The weighted digraph whose vertices are $x_1, ..., x_n$ and with only one arrow oriented from $x_i$ to $x_j$ to each non-zero product $x_i*x_j=\sum_{k=1}^n c_{ij}^k x_k$ with weight $(c_{ij}^1, ..., c_{ij}^n)$ is called operands digraph and represented by $G_O(A)$. 
\end{definition}

\begin{example}{} \label{example_OGraph}
\rm
Let $A$ be a 4-dimensional algebra given by the following table of multiplication.
$$A:\begin{cases}
[x_1,x_2]= 2x_3-x_1,\\
[x_1,x_1]= - x_2,\\
[x_2,x_3]= x_4,\\
[x_3,x_1]= 2x_1.
\end{cases}$$    

The operands digraph of $A$ is shown in Figure \ref{fig:grafo1}.

\begin{figure}[H] 
\centering 
        \includegraphics[width=0.25\linewidth]{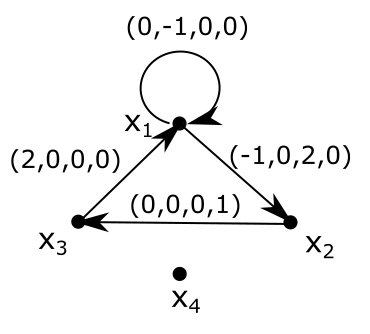}
        \caption{$G_O(A)$}
\label{fig:grafo1}
\end{figure}
\end{example}

Thanks to this definition, we can define the main digraph in this section: the union digraph.
\begin{definition} Let $(A,*)$ be an algebra over a field  $\mathbb{F}$ and $\{x_1,\dots, x_n\}$ a basis of $A$.

The union digraph of $A$ (denoted $G_U(A)$) is defined as the digraph with coloured edges whose vertices are $\{x_1,\dots, x_n\}$ and edges 
as follows.

\begin{itemize}

\item The set of edges in $G_U(A)$ is the union of the sets of edges in $G_O(A)$, $G_R(A)$ and $G_L(A)$. 
\item The weight of each edge is a vector with $3n$ coordinates, where each $n$ coordinates correspond to the weight of this edge in $G_O(A)$, $G_R(A)$ and $G_L(A)$) respectively. If an edge of the union digraph does not appear in one the digraphs $G_O(A)$, $G_R(A)$ or $G_L(A)$) then, the weight of this edge will have the corresponding $n$ zeros.
\item The colouring of each edge is represented by a vector in ${\mathbb{Z}_2}^3$ where the digit $1$ in each coordinate indicates the existence of this edge in $G_O(A)$, $G_R(A)$ or/and $G_L(A)$.
\end{itemize}
\end{definition}

\begin{example}{} \label{example_GraphU}
\rm
Let $A$ be the algebra in Example \ref{example_OGraph}, the union digraph of $A$ is shown in Figure \ref{fig:grafo_union}, where the red coloured label represents the image by colour function of each edge.


\begin{figure}[H]
\centering 
        \includegraphics[width=0.5\linewidth]{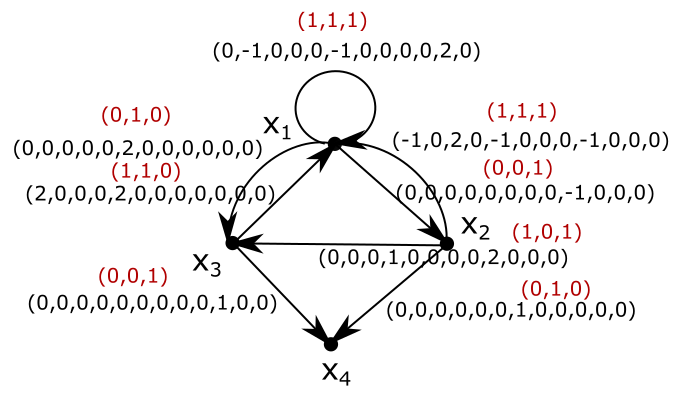}
        \caption{$G_U(A)$.}
\label{fig:grafo_union}
\end{figure}

\end{example}

\begin{observation}
\rm
We will use the notation $\pi_i(v)$ to denote the projection on the $i$-th coordinate of the vector $v$.
\end{observation}

Directly from definition of the union digraph, we have a fundamental result that will be useful for the main result of this section. In Lemma \ref{teo:celula_uniao} one can note the fact that the union digraph is related to an algebra makes that the digraph has a specific structure.

\begin{lemma}
\label{teo:celula_uniao}

Let $G$ be a weighted digraph with weight function $f:A(G) \to \mathbb{F}^3n$ and colour function $\Xi:A(G)\to {\mathbb{Z}_2}^3 $. If $G$ is the union digraph of an $\mathbb{F}$-algebra $A$, then each edge of $G$ belongs to a subdigraph with the configurations of the Figure \ref{fig:triangulo_uniao}.

\begin{figure}[H]
    \centering
    \includegraphics[width=0.25\textwidth]{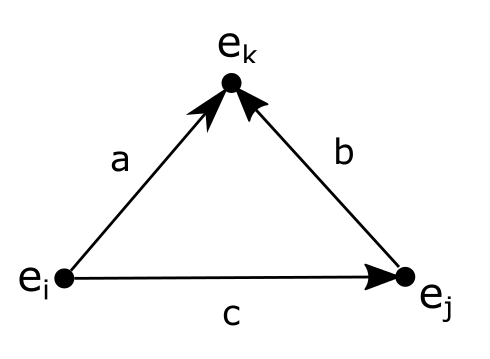} 
    \caption{General model to cell of the union digraph}
    \label{fig:triangulo_uniao}
\end{figure}
where \begin{equation} \pi_2(\Xi(a))=\pi_3(\Xi(b))=\pi_1(\Xi(c))=1 \text{ .} \end{equation} 
\end{lemma}

\begin{proof}
This result is trivially obtained from the definitions of the digraphs $G_R(A)$, $G_L(A)$, $G_O(A)$ and $G_U(A)$. In fact, each product $e_i * e_j = ... + c_{ij}^k e_k + ...$ is represented by one edge at $G_O(A)$ from $e_i$ to $e_j$, one edge at $G_R(A)$ from $e_i$ to $e_k$ and one edge at $G_L(A)$ from $e_j$ to $e_k$. Therefore, they are mutually dependent in the sense that one can not exist without the others.
\end{proof}

\begin{observation}
\rm
If an element of a fixed basis of $A$ plays more than one role in the product $e_i * e_j = ... + c_{ij}^k e_k + ...$, we need to rewrite the configuration of Figure \ref{fig:triangulo_uniao} as those in Figure \ref{fig:colorcell_alt}.

\begin{figure}[H]
\begin{subfigure}[Model to cell of the union digraph when $e_i = e_j$]{ \centering 
        \includegraphics[width=3.2cm]{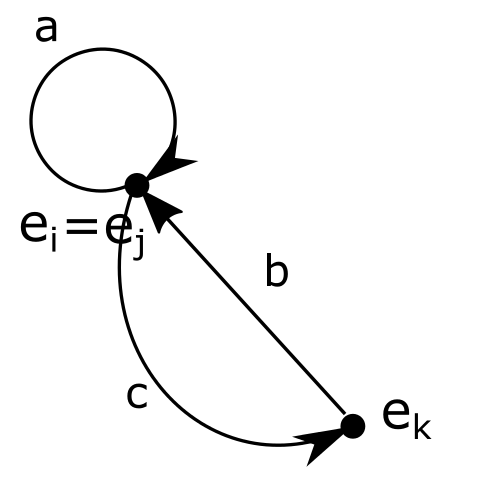}
        \label{subfig:colorcell_alt1} }
    \end{subfigure} 
    \hspace{0.7cm}
    ~ \hspace{0.7cm}
    \begin{subfigure}[Model to cell of the union digraph when $e_j = e_k$]{ \centering 
        \includegraphics[width=3.2cm]{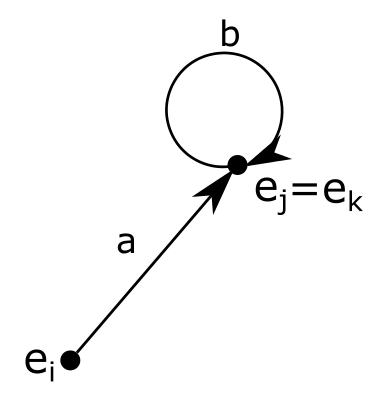}
        \label{subfig:colorcell_alt2}}
    \end{subfigure}
    ~ \\
    \begin{subfigure}[Model to cell of the union digraph when $e_i = e_k$]{ \centering 
        \includegraphics[width=2.8cm]{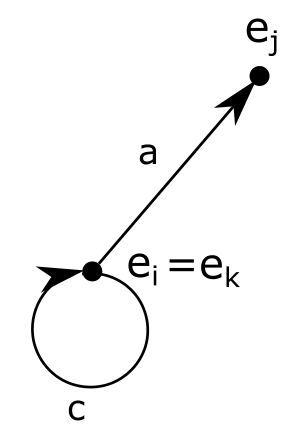}
        \label{subfig:colorcell_alt3}}
    \end{subfigure}
    \hspace{0.7cm}
    ~ \hspace{0.7cm}
    \begin{subfigure}[Model to cell of the union digraph when $e_i = e_j = e_k$]{ \centering 
        \includegraphics[width=2.8cm]{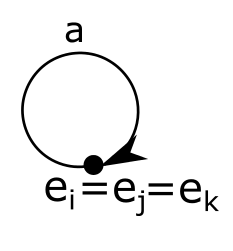}
        \label{subfig:colorcell_alt4}}
    \end{subfigure}
    \caption{Alternative models to singular cases to cell of the union digraph}
\label{fig:colorcell_alt}
\end{figure}

Note that similarly to the general case, these cells have the following restrictions:
\begin{itemize}
    \item In Figure \ref{subfig:colorcell_alt1},  $\pi_2(\Xi(a))=\pi_3(\Xi(b))=\pi_1(\Xi(c))=1$.
    \item In Figure \ref{subfig:colorcell_alt2}, $\pi_2(\Xi(a))=\pi_3(\Xi(b))=\pi_1(\Xi(b))=1$.
    \item In Figure \ref{subfig:colorcell_alt3}, $\pi_2(\Xi(a))=\pi_1(\Xi(c))=\pi_3(\Xi(c))=1$.
    \item In Figure \ref{subfig:colorcell_alt4}, $\pi_2(\Xi(a))=\pi_1(\Xi(a))=\pi_3(\Xi(a))=1$.
\end{itemize}
\end{observation}

It is worthwhile to note that not every digraph with these properties and colour function can be associated with a weight function such that it is a union digraph of some algebra. The next example shows this statement. 

\begin{example}{}
\rm
Let $G$ be the digraph coloured by function $\Xi:A->{\mathbb{Z}_2}^3 $ represented in Figure         \ref{fig:triangulo_uniao_example}.
\begin{figure}[H]
    \centering
    \includegraphics[width=0.3\textwidth]{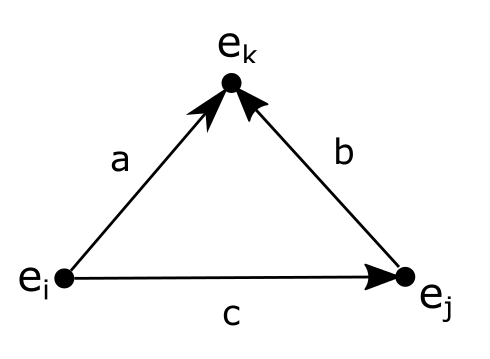} 
    \caption{General model to cell of the union digraph.}
    \label{fig:triangulo_uniao_example}
\end{figure}
where $\Xi(a)=(0,1,0)$ ; $\Xi(b)=(1,0,1)$ and  $\Xi(c)=(1,0,0)$.

It is easy to check that there is no algebra with this digraph associated. Otherwise, the edge $b$ would represent a non-null product $e_j * e_k$ since that $\pi_1(\Xi(b))=1$, but there are no edges where $e_k$ represents this product in $G_L(A),$ giving rise to a contradiction.
\end{example}

Therefore, the condition of Lemma \ref{teo:celula_uniao} is not enough to give a characterization.
Note that it occurs due to the presence of an extra colour on an edge which only satisfies the condition of Lemma \ref{teo:celula_uniao} when it is playing the role of another colour. Thus, the lemma can be rewritten in order to find a characterization as follows.

\begin{theorem}
Let $G$ be a digraph coloured by the function $\Xi:A(G)->{\mathbb{Z}_2}^3 $ and with $n$ vertices. If the following statements are true to any $\alpha \in A(G) $:
\begin{eqnarray*}
\pi_1(\Xi(\alpha))=1 \Rightarrow \text{ exists a subdigraph following the settings in Figure \ref{fig:triangulo_uniao} in which $\alpha$ plays the role of $c$;} \\
\pi_2(\Xi(\alpha))=1 \Rightarrow \text{ exists a subdigraph following the settings in Figure \ref{fig:triangulo_uniao} in which $\alpha$ plays the role of $a$;} \\
\pi_3(\Xi(\alpha))=1 \Rightarrow \text{ exists a subdigraph following the settings in Figure \ref{fig:triangulo_uniao} in which $\alpha$ plays the role of $b$,} \\
\end{eqnarray*}

then there exists a weight function $f:A(G) \to \mathbb{F}^n$ that can be associated to $G,$ such that this new digraph is the union digraph of an algebra $A.$
\end{theorem}

\begin{proof}
The proof is given by construction. 

Let $A$ be an algebra. For every vertex $e_i$, there must be an element  $x_i$ of the basis of $A$. 
On the other hand, to each edge $\alpha$ oriented from $e_i$ to $e_j$ playing the role of $c$ in configuration of Figure \ref{fig:triangulo_uniao}, let $\{e_{\alpha_1},e_{\alpha_2}, ..., e_{\alpha_p}\}$ be the list of vertex playing the role of $e_k$ in any subdigraph such that $\alpha$ plays the role of $c$.
Define the product $x_i * x_j = \sum_{t=1}^{p}x_{\alpha_t}$.

Since all edges are in this configuration by hypotesis, each edge in $A(G)$ occurs at the union digraph of $A$, and, by construction, it has the same orientation and colour of edge than in $G$. 
\end{proof}

\begin{corollary}
Let $G$ be a digraph with $n$ vertices. If each edge of $G$ belongs to (at least) one subdigraph with the configuration of Figure \ref{fig:triangulo_uniao}, then there exists a weight function and a colour function such that $G$ eqquiped with them is the union digraph of an algebra.
\end{corollary}
\begin{proof}
Let $ \Xi:A(G)->{\mathbb{Z}_2}^3 $ be a function colour. 
On the one hand, for every edge playing the role of $c$ in this configuration, let define $\pi_1(\Xi(c))=1$. On the other hand,  for every $a_i$ and $b_i$ living in this configuration with $c$, let define $\pi_2(\Xi(a_i))=\pi_3(\Xi(b_i))=1$. Finaly, set the remaining coordinates as zero.

Since each edge belongs to one of these configurations, all of them have a defined colour and belong to one subdigraph with the configuration of Figure \ref{fig:triangulo_uniao}. Therefore, by Theorem \ref{teo:celula_uniao}, there is an algebra whose union digraph is isomorphic to $G.$ 

\end{proof}

Note that could be more than one algebra associated to this digraph $G.$

As we said at the beginning of this section, we want to give one step more and analyse the structure of a graph (without neither orientation or weight) in order to study its associated algebra. That is possible thanks to the previous results and mainly, to the following corollary.

\begin{corollary}
Let $G$ be a non-oriented and non-weighted graph. If each edge in $G$ belongs to a cycle with length three, then there exists an orientation of edges of $G$, a weight function, and a colour function such that $G$ equipped with them is the union digraph of (at least) one algebra.
\end{corollary}
\begin{proof}
Thanks to the previous corollary, it remains to show that it is possible to orient the edges in a way that each of them belongs to a triple of vertices with the configuration given in Figure \ref{fig:triangulo_uniao}.

We will make a constructive proof. Firstly, select an edge and give it any orientation. Then, keep selecting edges and orienting them following the instructions below:
\begin{itemize}
\item Select one of the 3-cycles that contains this edge.
\item If the 3-cycle has one oriented edge, define any orientation to other edge.
\item If the 3-cycle has two oriented edges, it is always possible define an orientation to the third edge in order to satisfy the allowed configuration (see Figure \ref{fig:triangulo_uniao}).
\end{itemize}
In this orientation, each edge belongs to at least one cycle with the configuration in Figure \ref{fig:triangulo_uniao}. In fact, the possibilities to two oriented edges in a cycle with length three are shown in Figure \ref{fig:posibilidades}
\begin{figure}[H]
\begin{subfigure}[Possibility 1]{ \centering 
        \includegraphics[width=3.2cm]{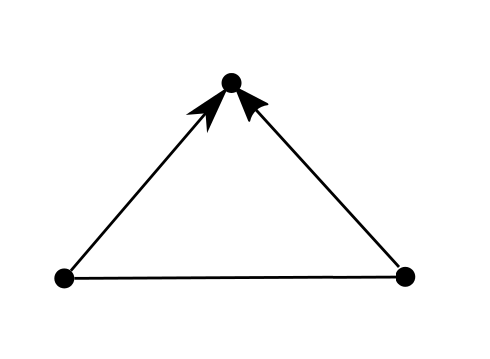}
       }
    \end{subfigure}
    \begin{subfigure}[Possibility 2]{ \centering 
        \includegraphics[width=3.2cm]{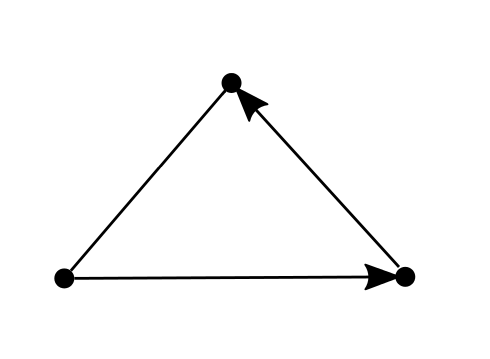}
        }
    \end{subfigure}
    \begin{subfigure}[Possibility 3]{ \centering 
    		\includegraphics[width=3.2cm]{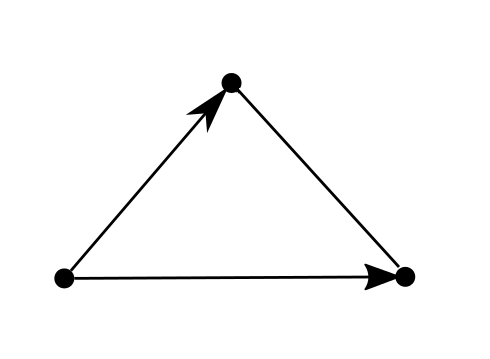}
        }
    \end{subfigure}
    \caption{Possibilities to 3-cycles with two edges oriented}
\label{fig:posibilidades}
\end{figure}

It is enough to apply the previous corollary to finish the proof. 


\end{proof}

Finally, let us see as some important algebraic properties of an algebra can be studied by analysing the union digraph. Previously, we need to introduce the fine sequence (a sequence of subdigraphs).

\begin{definition}
Let $G$ be the union digraph associated to the algebra $A$ and $\Xi$ its colour function.

The fine sequence of $G$ is the sequence of subdigraphs $G^{(1)}, G^{(2)}, ..., G^{(n)}$ started at $G^{(1)}=G$ and such that $G^{(i+1)}$ is obtained from $G^{(i)}$ by the following procedure:
\begin{itemize}
	\item Remove all vertices that are not the endpoint of an edge which image by the function $\pi_2 \circ \Xi$ is equal to 1.
    \item Set as zero the coordinates of images of function $\Xi$ that does not satisfy the Lemma \ref{teo:celula_uniao}.
    \item Remove all edges such that the image by function $\Xi$ is $(0,0,0)$.  
    \item If $G^{(i+1)}$ is not longer a digraph (if there is no more vertices), then make $G^{(i+1)}=G^{(i)}$. 
\end{itemize}
\end{definition}

\begin{observation}
\rm
The fine sequence is non--increasing with the ordering relation $\subseteq$ (to edge set or  to vertex set). Since the digraph is finite, the sequence stabilizes.
\end{observation}

It is worthwhile to point out the importance of the fine sequence, because it allows to describe a method to recognise soluble Leibniz algebras. We formalize this result in the following theorem.

\begin{theorem}
Let $G$ be the union digraph of an algebra $A$. If the fine sequence stabilizes in a digraph without edges, then the derivated series of $A$ stabilizes in zero.
\end{theorem}
\begin{proof}
Let $B=\{e_1, \dots, e_n \}$ be a basis of $A$ and $y=\sum_{i=1}^n \alpha_i e_i$ an element of $A^{(k)}.$ 

Firstly, we will prove that for any $i \in [1,n]$ such that  $\alpha_i \neq 0,$ $e_i \in V(G^{(k)}).$ This proof will be done by induction on $k.$

For $k=1$ the assertion is clear and suppose that it is true for $k.$ On one hand, recall that in step $k$ of  the fine sequence, vertices of $B$ that do not belong to the linear combinations of elements of $A^{(k)}$ are excluded. On the other hand, the colours of edges generated by these vertices can be removed. 

From definition of the fine sequence we know that the remove vertices from $G^{(k)}$ to $G^{(k+1)}$ are those whose indegree are zero. By induction, the product of any elements of $A^{(k)}$ is represented in $G^{(k)}.$ Hence, the vertices with indegree zero are vertices that are not in the linear combinations of $A^{(k)}*A^{(k)}$ and then, they do not belong to $A^{(k+1)}.$ Therefore $A^{(k+1)} \subset V(G^{(k+1)}).$

Finally, if the fine sequence stabilizes at a digraph $ G^{(n)} $ without edges, there are no products between the elements represented by remaining vertices, therefore $A^{(n+1)}=\{0\}$.
\end{proof}


\begin{observation}
\rm
Even though this algorithm covers all cases of Theorem \ref{solvable}, it is important to note that this algorithm is not a characterization, since there are algebras whose derivated series estabilize in zero but the fine sequence of its associated digraph does not estabilize in a digraph without edges. Let us see that in the following example. \end{observation}

\begin{example}{} \label{example_FineSeq_No}
\rm
Let $A$ be the 4-dimensional algebra given by the following table of multiplication.

$$A:\begin{cases}
[x_1,x_2]= x_3+x_4,\\
[x_2,x_1]= -x_3-x_4,\\
[x_3,x_4]= x_3+x_4,\\
[x_4,x_3]= -x_3-x_4.
\end{cases}$$    

The derivated series of $A$ is given by:
\begin{eqnarray*}
A^{[1]}&=&\langle x_1, x_2, x_3, x_4 \rangle;\\
A^{[2]}&=&\langle x_3,x_4 \rangle;\\
A^{[3]}&=&\{0 \};
\end{eqnarray*}

The union digraph of $A$ and its fine sequence are showed in Figures \ref{fig:casifinal} and \ref{fig:final}. For the sake of simplicity, we will omit the weights of the edges.
\begin{figure}[H]
\centering 
        \includegraphics[width=0.3\linewidth]{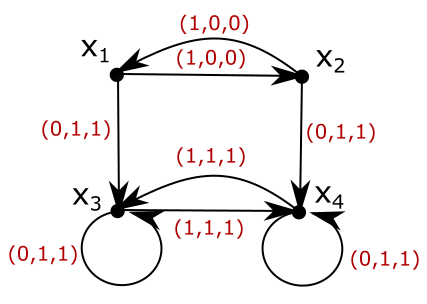}
        \caption{$G_U(A)^{(1)} = G_U(A)$}
\label{fig:casifinal}
\end{figure}
\begin{figure}[H]
\centering 
        \includegraphics[width=0.3\linewidth]{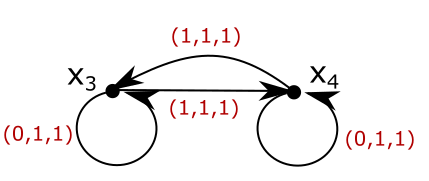}
        \caption{$G_U(A)^{(2)} = G_U(A)^{(n)} \ \forall n \geq 2.$}
\label{fig:final}
\end{figure}
It is clear that the hypothesis of the previous theorem is not obtained, because even the derivated series of $A$ stabilises in zero, the fine sequence stabilises in $G_U(A)^{(2)}$, which is a digraph with edges, as we can see in the Figure \ref{fig:final}. 
\end{example}

\section{Concluding remarks and open questions}
In this paper we have introduced new and natural methods to associate (directed) graphs to algebras. Regarding previous works as that of Carriazo et al. \cite{Carriazo}, our method has two advantages: on one hand, we have not the restriction of considering only Lie algebras and our graphs can be associated to any algebra. And, on the other hand, by checking very simple properties of graphs (as to be a directed acyclic graph, fact that can be checked algorithmically in linear time) we can deduce important properties of the original algebra (as to be a nilpotent algebra in the case of directed acyclic graph).  Additionally, in our weighted digraphs, we store all the information of the original algebra, thus this can be rebuild completely from the discrete structures.

Of course, this work must be considered as a first step toward a more general goal. So, we thing that more properties of the algebras can be extracted by examining the associated digraphs. 

Finally, it is known that the classification of algebras up to isomorphism is a extremely difficult task, so we know only results in some particular cases and in low dimensions (see \cite{Casas}, \cite{Canete}, and  \cite{rakhimov} for more details). Although, the problem of graph isomorphism is not known to be a polynomial problem, in practice, it is much easier to solve that the equivalent in algebraic structures, so we think that it can be worthy to study what kind of partition is obtained if we group the algebras with the same (or similar under some discrete property) associated digraphs.

\vspace{1cm}

\noindent{{\bf Acknowledgement:}} The authors want to thank  Antonio Cañete, Isabel Fernández, Antonio Gonz\'alez, and Juan Vicente Gutiérrez for a very fruitful help.

\end{document}